\documentclass[11]{siamltex}

\usepackage{datetime}
\settimeformat{ampmtime}
\usdate

\newcommand{\nwc}{\newcommand}
\nwc{\draftdate}{\today}
\usepackage[pdftex]{graphicx}
\usepackage{psfrag}
\usepackage{wrapfig}
\usepackage{epsf}
\usepackage{amsmath,amssymb}
\usepackage[font={footnotesize}]{caption} 
\usepackage{url}
\usepackage[mathscr]{euscript}

\newtheorem{thm}{Theorem}[section]
\newtheorem{coro}[thm]{Corollary}

\newcommand{\barint}{\hbox{$\int$\kern-0.75\intwidth
\vrule width 0.5\intwidth height 2.4pt depth -2pt\kern0.25\intwidth}}
\newlength\intwidth
\setbox0=\hbox{$\int$}
\intwidth=\wd0

\newcommand\avint{\hbox{\hbox{$\displaystyle \int$}\hbox{\kern-.9em{$-$}}}}

\newcommand\smavint{\hbox{\hbox{$\int$}\hbox{\kern-.75em{$-$}}}}

\nwc{\st}{^{\mbox{\it st}}}

\nwc{\qref}[1]{(\ref{#1})}
\nwc{\veloc}{v}
\nwc{\rhoc}{\beta}
\nwc{\hl}{\hat{L}}

\def\Xint#1{\mathchoice
{\XXint\displaystyle\textstyle{#1}}%
{\XXint\textstyle\scriptstyle{#1}}%
{\XXint\scriptstyle\scriptscriptstyle{#1}}%
{\XXint\scriptscriptstyle\scriptscriptstyle{#1}}%
\!\int}
\def\XXint#1#2#3{{\setbox0=\hbox{$#1{#2#3}{\int}$}
\vcenter{\hbox{$#2#3$}}\kern-.51\wd0}}

\def\dashint{\Xint-}
\nwc{\intRp}{\int_0^\infty}
\nwc{\aint}{\dashint}
\nwc{\aaint}{\dashint}

\newcommand{\cA}{{\cal A}}

\newcommand{\cF}{{\cal F}}

\newcommand{\R}{\mathbb R}

\newcommand{\be}{\begin{eqnarray}}
\newcommand{\ee}{\end{eqnarray}}

\newcommand{\ben}{\begin{eqnarray*}}
\newcommand{\een}{\end{eqnarray*}}

\title{}
\author{}
\date{}

\numberwithin{equation}{section}
\usepackage{indentfirst}
\usepackage{stmaryrd}

\begin{document}
	
\title{A Nonnegative Weak Solution to the Phase Field Crystal Model with Degenerate Mobility}
\author{
	Toai Luong
	\thanks{Corresponding author, 
		Department of Mathematics, The University of Tennessee, Knoxville, TN 37996, USA. 
		Email: tluong4@utk.edu.
	} \and 
	Steve Wise
	\thanks{
		Department of Mathematics, The University of Tennessee, Knoxville, TN 37996, USA. 
		Email: swise1@utk.edu
	} 
}

\date{\today}
\maketitle

\begin{abstract}
	Phase field crystal is a model used to describe the behavior of crystalline materials at the mesoscale. 
	In this study, we investigate the well-posedness of a phase field crystal equation 
	subject to a degenerate mobility $M(u)$ that equals zero for $u\leq 0$. 
	First, we prove the existence of a weak solution to a phase field crystal equation with non-degenerate cutoff mobility. 
	Then, assuming that the initial data $u_0(x)$ is positive, 
	we establish the existence of a nonnegative weak solution to the degenerate case. 
	Such solution is the limit of solutions corresponding to non-degenerate mobilities. 
	We also verify that such a weak solution satisfies an energy 
	dissipation inequality. 	
\end{abstract}


\section{Introduction}
In this paper, we consider a phase field crystal (PFC) model. Let $u:\Omega\to [0,\infty)$ be an $\Omega$--periodic unary atom density, where $\Omega=[0,2\pi]^d\subset \R^d$, $1 \leq d \leq 3$. The free energy for the state defined by $u$ is
\begin{align}\label{pfc-ener}
	\cF(u)=\int_\Omega \left[ W(u) + \kappa\left(\frac{1}{2}u^2 - |\nabla u|^2 + \frac{1}{2}|\Delta u|^2\right) \right] dx,
\end{align}
where $W:[0,\infty)\to [W_0,\infty)$ is a homogeneous free energy density, $W_0>-\infty$, and $\kappa>0$ is a parameter. The PFC equation  is a conserved gradient flow with respect to the PFC energy functional $\cF(u)$ and is written as 
	\begin{align}
	\label{pfc-eqn1} 
u_t&=\nabla\cdot(M(u)\nabla\omega), \quad \mbox{in} \  \Omega_T:=\Omega\times(0,T), 
	\\	
	\label{pfc-eqn2}
\omega &=W'(u) + \kappa(u + 2\Delta u + \Delta^2 u), \quad \mbox{in} \  \Omega_T, 
\end{align}	
where $T>0$ is the final time.

From a mathematical point of view, this equation and its variants have been analyzed in \cite{Miranville2015, Wu2022, Grasselli2014, Grasselli2015, Conti2016} and references therein. 
In particular, Miranville studied the existence and uniqueness of variational solutions for a PFC model with constant mobility $M(u) \equiv 1$ and logarithmic (singular) nonlinear terms \cite{Miranville2015}. 
Wu and Zhu theoretically and numerically analyzed the well-posedness of a square PFC model in the three-dimensional case \cite{Wu2022}. 
The significant difference between our result and theirs is that we are studying the PFC equation subject to a non-constant mobility $M(u)$ that is dependent on $u$.

	\subsection{Derivation}
	
In this section, we quickly derive the PFC model to motivate the form that we investigate. See, for example, the nice paper by Archer~et al.~\cite{archer2019} for all of the possible modeling choices. The phase field crystal (PFC) model was introduced in~\cite{elder02,elder04} as continuum description of solidification in a unary material. It was formulated as a mass conservative version of the classical Swift-Hohenberg equation, but, later, the model was re-derived,  via certain reasonable simplifications, from the dynamical density functional theory (DDFT)~\cite{elder07}. (See~\cite{archer2019} for a more in-depth derivation from the DDFT framework.)   In particular, assuming a constant, uniform temperature field, one expresses the dimensionless Helmholtz free energy density via
	\[
F(\rho) = \int_\Omega\left[ f(\rho) + \frac{\kappa}{2}(\rho-1) \mathcal{C} (\rho-1) \right]dx,
	\]
where $\Omega$ is some spatial domain of interest; $\rho:\Omega \to [0,\infty)$ is the number density field of the unary material in $\Omega$; $\kappa >0$ is a positive dimensionless constant; $f$ is the homogeneous Helmholtz free energy density; and $\mathcal{C}$ is a symmetric, potentially nonlocal, two-point correlation operator. Here we have taken the state $\rho = \rho_o = 1$ as the dimensionless reference density.

The homogeneous free energy density, $f$, is often taken to satisfy an ``ideal gas" model:
	\[
f(\rho) = \rho \ln\left( \rho \right) -  \rho .
	\]
Note that this energy density has a global minimum at the reference state $\rho = 1$. Often, one makes a (Taylor) polynomial approximation of the logarithmic term about the reference density to make the model more tractable. However,  the singular nature of the logarithmic term contributes to the positivity of the solutions, and this is an important feature in the numerical and PDE analyses.  At constant temperature, one can argue that the dynamics of the model should satisfy a diffusion-dominated mass conservation equation of the form
	\begin{equation}
\partial_t \rho = -\nabla\cdot \mathbf{J} , \quad \mathbf{J} = - \rho \nabla \mu,	
	\end{equation}
where $\mathbf{J}$ is the diffusion flux; and $\mu$ is the chemical potential:
	\begin{equation}
\mu := \delta_\rho F = \log (\rho) + \kappa \mathcal{C} (\rho-1) ,
	\label{eqn:chemical-potential}
	\end{equation}
where we have assumed, for simplicity, that the boundary conditions are periodic. Observe that
	\[
\nabla\cdot (\rho \nabla(\log(\rho))) = \Delta \rho,
	\]
which suggests that the model incorporates standard particle diffusion at leading order. For this derivation, we will assume that $\mathcal{C}$ is a differential operator of the form
	\[
\mathcal{C}(\varrho) =  \gamma \varrho + 2\Delta \varrho + \Delta^2 \varrho ,
	\]
where $\gamma\in \mathbb{R}$ is a dimensionless parameter. This chosen form  promotes the formation of spatially oscillatory density fields and is most commonly used. To sum up, the chemical potential is modeled as
	\begin{equation}
\mu = \log (\rho)  +\kappa\gamma (\rho-1) + \kappa\left(2\Delta\rho +\Delta^2\rho \right) .
	\label{eqn:chemical-potential-2}
	\end{equation}

As a consequence of our model assumptions, the total free energy is dissipated as the system evolves toward equilibrium, and the dissipation rate is
	\[
d_t F = -\int_\Omega \rho|\nabla \mu|^2\, dx \le 0.
	\]
Of course, is be necessary to justify the property that $\rho >0$ (or at least $\rho\ge0$) point-wise for the model to make sense.

As we indicated, it is most typical in the physics literature to use a Taylor polynomial approximation of the logarithmic free energy density, around the dimensionless reference state $\rho =\rho_o = 1$, of the form
	\begin{equation}
 \rho\ln(\rho) -\rho  \approx -1 + \frac{1}{2} (\rho-1)^2 - \frac{1}{6}(\rho-1)^3 + \frac{1}{12} (\rho-1)^4 .
	\end{equation}
This approximation regularizes the singular nature of the ideal gas law, but also it removes the singular free energy barrier against states that have negative density regions.

Note that we are keeping the degenerate mobility and approximating the ideal gas law with a polynomial. Archer~et al.~\cite{archer2019} have argued that if, on the other hand, one approximates the degenerate mobility with a positive constant, then one is absolutely forced to replace the logarithmic potential (the ideal gas law) by the approximating Taylor polynomial. Otherwise, the model predictions are unphysical.

With the simplifications described above, we have the following PFC model. For each $\Omega$--periodic state function $u:\Omega \to [0,\infty)$, the free energy is
	\begin{equation}
\mathcal{F}(u) = \int_\Omega \left[ f_0(u) +  \frac{\kappa}{2}(u-1) \left( \gamma (u-1) + 2\Delta (u-1) + \Delta^2 (u-1) \right) \right]  dx,
	\end{equation}
where
	\[
 f_0(u) = -1 + \frac{1}{2} (u-1)^2 - \frac{1}{6}(u-1)^3 + \frac{1}{12} (u-1)^4.
	\]
The conserved gradient flow describing the evolution of the state function $u$ is
	\begin{align*}
\partial_t u & = -\nabla\cdot \left( u \nabla \omega \right) ,
	\\
\omega & = f_0'(u) +\kappa \left( \gamma(u-1) + 2\Delta u+ \Delta^2 u\right),
	\end{align*} 
which matches the form that we introduced earlier, after defining $W(u)$ to be
	\[
W(u) = f_0(u) + \frac{\kappa\gamma}{2}(u-1)^2 - \frac{\kappa}{2}u^2.	
	\]

\subsection{Main result}
In this paper, we  model the mobility $M(u)$ as
	\begin{align}
	\label{M(u)}
M(u)=
	\begin{cases} 
u , &  u>0, 
	\\
0 , &  u\leq 0, 
	\end{cases}
	\end{align}
which degenerates when $u \leq 0$. We assume that the potential $W\in C^2(\R)$ satisfies the following growth conditions:
\begin{align}
	\label{grow-1}b_1 z^{2m}-b_2\leq&W(z) + \frac{\epsilon}{2}z^2\leq b_3 z^{2m}+b_4, \\
	\label{grow-2}|&W'(z)|\leq b_3|z|^{2m-1}+b_4, \\
	\label{grow-3}b_1 z^{2m-2}-b_2\leq&W''(z)\leq b_3 z^{2m-2}+b_4,
\end{align} 
for all $z\in\R$, where $m>1$ is an integer, $0<\epsilon<\kappa$, and $b_1, b_2, b_3, b_4$ are positive constants. A simple example of such a potential $W$ is
	\begin{align*}
W(u)=\frac{1}{4}(u-1)^4 - \frac{\epsilon}{2}(u-1)^2,
	\end{align*}
where $\epsilon\in \mathbb{R}$. This satisfies the above growth conditions with  $m=2$.

Our analysis used a framework similar to \cite{Elliott-CH-sln,DaiCH-sln,LuongFCH-sln}.  
First, we approximate the degenerate mobility $M(u)$ using a nondegenerate mobility $M_\theta(u)$ defined  by
\begin{align}\label{M-theta}
	M_\theta(u)=\begin{cases}
				u, & u>\theta, \\
				\theta, & u\leq\theta,
				\end{cases}
\end{align}
for any $\theta\in(0,1)$, 
and show that the equation \qref{pfc-eqn1}--\qref{pfc-eqn2} with the mobility $M_\theta(u)$ has a sufficiently regular weak solution $u_\theta$,  
that is, we prove the following theorem.

	\begin{thm}\label{thm-main1}
Let $u_0\in H^2(\Omega)$. For any given constant $T>0$, there exists a function $u_\theta$ that satisfies the following conditions:
	\begin{enumerate}	
	\item[(i)] 
For any $0<\alpha<1/2$,
	\begin{align*}
u_\theta & \in L^2(0,T;H^5(\Omega))\cap L^2(0,T;C^{3,\alpha}(\bar{\Omega}))
	\\
& \quad \cap C([0,T];H^1(\Omega))\cap C([0,T];C^{0,\alpha}(\bar{\Omega})),
	\end{align*}

	\item[(ii)]  
$\partial_tu_\theta\in L^2(0,T;H^{-2}(\Omega))$.
	\item[(iii)] 
$u_\theta(x,0)=u_0(x)$, for all $x\in\Omega$.
	\item[(iv)] 
$u_\theta$ solves the PFC equation \qref{pfc-eqn1}--\qref{pfc-eqn2} in the following weak sense:
	\begin{align}
&\int_0^T\left\langle \partial_tu_\theta,\xi\right\rangle_{H^{-2}(\Omega),H^2(\Omega)}dt
	\nonumber
	\\ 
&=-\int_0^T\int_\Omega M_\theta(u_\theta)(W''(u_\theta)\nabla u_\theta + \kappa\nabla u_\theta + 2\kappa\nabla\Delta u_\theta + \kappa\nabla\Delta^2 u_\theta)\cdot\nabla\xi \;dxdt,
	\nonumber
	\\
& \mbox{ }
	\label{PFC-w2}
	\end{align}
for all $\xi\in L^2(0,T;H^2(\Omega))$. In addition, for any $t\geq 0$, the following energy inequality holds:
	\begin{align} 
& \int_\Omega W(u_\theta(x,t)) + \kappa\left(\frac{1}{2}|u_\theta(x,t)|^2 - |\nabla u_\theta(x,t)|^2 + \frac{1}{2}|\Delta u_\theta(x,t)|^2\right) dx  
	\nonumber
	\\ 
& \quad +\int_0^t\int_\Omega M_\theta( u_\theta(x,\tau)) \bigg|W''(u_\theta(x,\tau))\nabla u_\theta(x,\tau) 
	\nonumber
	\\
& \quad + \kappa\left(\nabla u_\theta(x,\tau) + 2\nabla\Delta u_\theta(x,\tau) + \nabla\Delta^2 u_\theta(x,\tau) \right)\bigg|^2 \; dxd\tau 
	\nonumber
	\\
&\leq\int_\Omega W(u_0) + \kappa\left(\frac{1}{2}u_0^2 - |\nabla u_0|^2 + \frac{1}{2}|\Delta u_0|^2\right) dx.
	\label{ener-ineq1}
	\end{align}
	\item[(v)]
If $u_0(x)>0$, for all $x\in\Omega$, then
	\begin{align}
	\label{bnd-neg1}
\mathop{\mathrm{ess\, sup}}_{0\leq t \leq T}\int_\Omega\left|(u_\theta(x,t))_-+\theta\right|^2 \; dx \leq C(\theta^2+\theta+\theta^{1/2}),
	\end{align}
where $(u_\theta)_-=\min\{u_\theta,0\}$, and $C$ is a generic positive constant that may depend on $d,T,\Omega$, $b_1, b_2, b_3, b_4$, $m, \kappa,\epsilon$ and $u_0$, but not on $\theta$.
	\end{enumerate}
	\end{thm}

Next, we consider the limit of $\{u_\theta\}$ as $\theta\to  0$. 
We show that the limiting function $u$ of $\{u_\theta\}$ exists and is a weak solution to the PFC equation \qref{pfc-eqn1}--\qref{pfc-eqn2} with the mobility $M(u)$ defined by \qref{M(u)}. 
Moreover, if the initial data $u_0(x)$ is positive in $\Omega$, using the estimate \qref{bnd-neg1} in part (v) of Theorem \ref{thm-main1}, we can show that such weak solution is nonnegative in $\Omega_T$. 
We will prove the following result.

\begin{thm}\label{thm-main2}
	Assume that $u_0\in H^2(\Omega)$. 
	For any given constant $T>0$, there exists a function $u$ that satisfies the following conditions:
	
	\begin{enumerate}
	\item[(i)] 
For any $0<\alpha<1/2$,
	\[
u\in L^\infty(0,T;H^2(\Omega))\cap C([0,T];H^1(\Omega))\cap C([0,T];C^{0,\alpha}(\bar{\Omega})).
	\] 
	\item[(ii)] 
$\partial_tu\in L^2(0,T;H^{-2}(\Omega))$.
	\item[(iii)] 
$u(x,0)=u_0$.
	\item[(iv)] 
$u$ can be considered as a weak solution for the PFC equation \qref{pfc-eqn1}--\qref{pfc-eqn2} in the following weak sense:
	\begin{enumerate}
	\item[(a)] 
Let $P$ be the set where $M(u)$ is not degenerate, that is,
	\begin{align}
P:=\{(x,t)\in\Omega_T:u(x,t)>0\}.
	\end{align}
There exist a set $B\subset\Omega_T$ with $|\Omega_T\backslash B|=0$ and a function $\Psi:\Omega_T\to \R^d$ satisfying $\chi_{B\cap P}M(u)\Psi\in L^2(0,T;L^{2d/(d+2)}(\Omega;\R^d))$, 
where $\chi_{B\cap P}$ is the characteristic function of $B\cap P$, such that
	\begin{align}
	\label{PFC-w3}
\int_0^T\left\langle\partial_tu,\xi\right\rangle_{H^{-2}(\Omega),H^2(\Omega)}dt=-\int_{B\cap P}M(u)\Psi\cdot\nabla\xi dxdt,
	\end{align}
for all $\xi\in L^2(0,T;H^2(\Omega))$.
	\item[(b)] 
Let  $\nabla\Delta^2u$ be the generalized derivative of $u$ in terms of distributions.  If $\nabla\Delta^2u\in L^q(U_T)$, where $U_T=U\times(0,T)$, 
for some open set $U\subset\Omega$ and some $q>1$ that may depend on $U$, then we have
	\begin{align*}
\Psi=W''(u)\nabla u + \kappa\nabla u + 2\kappa\nabla\Delta u + \kappa\nabla\Delta^2 u \quad \text{in } U_T.
	\end{align*}
	\item[(c)] 
For any $t\geq 0$, the following energy inequality holds:
	\begin{align}
&W(u(x,t)) + \kappa\left(\frac{1}{2}|u(x,t)|^2 - |\nabla u(x,t)|^2 + \frac{1}{2}|\Delta u(x,t)|^2\right)dx 
	\nonumber
	\\ 
& \quad +\int_{\Omega_t\cap B\cap P} M(u(x,\tau))|\Psi(x,\tau)|^2dxd\tau 
	\nonumber
	\\ 
&\leq\int_\Omega W(u_0) + \kappa\left(\frac{1}{2}u_0^2 - |\nabla u_0|^2 + \frac{1}{2}|\Delta u_0|^2\right) dx.
	\label{ener-ineq2}
	\end{align}
	\end{enumerate}
	\item[(v)] 
In addition, if $u_0(x)>0$, for all $x\in\Omega$, then $u(x,t)\geq 0$, for a.e. $x\in\Omega$ and all $t\in[0,T]$, 
		and $u(x,t)$ is not always zero in $\Omega_T$.
	\end{enumerate}
\end{thm}

\subsection{Notation}
In this paper, we use $C$ to denote a generic positive constant that may depend on $d,T,\Omega$, $b_1, b_2, b_3, b_4$, $m, \kappa,\epsilon$ and $u_0$, but nothing else, in particular not on $\theta$ and $N$. 
We also use $C_\theta$ to denote a generic positive constant that may depend on $d,T,\Omega$, $b_1, b_2, b_3, b_4$, $m, \kappa,\epsilon,u_0$ and $\theta$, but not on $N$.

The reamainder of this paper is organized as follows. 
In Section \ref{pos-mobi}, we prove Theorem \ref{thm-main1} using the Galerkin approximation. 
In Section \ref{dege-mobi}, we prove Theorem \ref{thm-main2}, which shows the existence of a weak solution to the PFC equation \qref{pfc-eqn1}--\qref{pfc-eqn2} with the degenerate mobility \qref{M(u)}.

\section{Proof of Theorem \ref{thm-main1}: Weak solution for the positive mobility case} \label{pos-mobi}
In this section we prove the existence of a weak solution to the PFC equation \qref{pfc-eqn1}--\qref{pfc-eqn2} 
with the nondegenerate mobility $M_\theta(u)$ defined by \qref{M-theta}.

\subsection{Galerkin approximation} \label{Galerkin}
Let $\{\phi_j:j=1,2,...\}$ be the normalized eigenfunctions in $L^2(\Omega)$, 
that is, $\|\phi_j\|_{L^2(\Omega)}=1$, of the eigenvalue problem:
\begin{align*}
	-\Delta u=\lambda u \text{ in } \Omega, \text{ and } 
	u \text{ is } \Omega\text{--periodic}. 
\end{align*} 
The eigenfunctions $\phi_j$ are orthogonal with respect to the $H^2(\Omega)$ and $L^2(\Omega)$ inner products. 
We assume that $\lambda_1=0$, which implies that $\phi_1(x)\equiv(2\pi)^{-d/2}$.

We consider the Galerkin approximation for the PFC equation \qref{pfc-eqn1}--\qref{pfc-eqn2}:
\begin{align}
	u^N(x,t)&=\sum_{j=1}^{N}c^N_j(t)\phi_j(x),\quad
	\omega^N(x,t)=\sum_{j=1}^{N}d^N_j(t)\phi_j(x),	\\	
	\label{ode1}\int_\Omega\partial_tu^N\phi_jdx&=-\int_\Omega M_\theta(u^N)\nabla\omega^N\cdot\nabla\phi_jdx \\	
	\label{ode2}\int_\Omega\omega^N\phi_jdx&=\int_\Omega(W'(u^N)\phi_j+\kappa u^N\phi_j-2\kappa\nabla u^N\cdot\nabla\phi_j+\kappa\Delta u^N \Delta\phi_j)dx \\	
	\label{ode3}u^N(x,0)&=\sum_{j=1}^{N}\left(\int_\Omega 	u_0\phi_jdx\right)\phi_j(x).
\end{align}
This system is equivalent to the following system of ordinary differential equations for $c^N_1,...,c^N_N$:
\begin{align}
	\label{ode4}\partial_tc_j^N=&-\sum_{k=1}^{N}d_k^N\int_\Omega M_\theta\left(\sum_{i=1}^{N}c^N_i\phi_i\right) \nabla\phi_k\cdot\nabla\phi_jdx, \\	
	\label{ode5}d_j^N=&\int_\Omega W'\left(\sum_{k=1}^{N}c^N_k\phi_k\right)\phi_jdx + \kappa c_j^N - 2\lambda_jc_j^N + \kappa\lambda_j^2 c_j^N, \\	
	\label{ode6}c_j^N(0)=&\int_\Omega u_0\phi_jdx.
\end{align}
Because the right hand side of \qref{ode4} depends continuously on $c^N_1,...,c^N_N$, 
according to the Peano existence theorem \cite{Peano1, Peano2}, 
the initial value problem \qref{ode4}--\qref{ode6} has a local solution. 
Lemma \ref{bnd-lem} gives the uniform bound for $c^N_1,...,c^N_N$, 
therefore, by the Picard--Lindel\"{o}f theorem (see \cite{CoddingtonODE}, page 12, Theorem 3.1), 
a global solution for the initial value problem \qref{ode4}--\qref{ode6} exists.

	\subsection{A Priori Estimates}
	\label{1st-estimate}
We will establish some necessary prior estimates in this section.
	\begin{lemma}
	\label{ini-lem}
Let $u^N$ be a solution of the system \qref{ode1}--\qref{ode3}. Then, for any $0\leq t\leq T$,
	\begin{align}
\int_\Omega u^N(x,t)dx=\int_\Omega u^N(x,0)dx.
	\end{align}
	\end{lemma}

	\begin{proof}
Taking $j=1$ in \qref{ode1} gives $\int_\Omega\partial_tu^Ndx=0$. 
Thus, $\int_\Omega u^N(x,t)dx=\int_\Omega u^N(x,0)dx$, for any $0\leq t\leq T$.
	\end{proof}

\begin{lemma}\label{bnd-lem}
	Let $u^N$ be a solution of the system \qref{ode1}--\qref{ode3}, then, for any $0\leq t\leq T$,
	\begin{align}
		\label{bnd-uN-Linf-H2}\|u^N\|_{L^\infty(0,T;H^2(\Omega))}&\leq C, \\
		\label{bnd-M-omN}\int_0^t\int_\Omega M_\theta(u^N(x,\tau))|\nabla\omega^N(x,\tau)|^2dxd\tau&\leq C.
	\end{align}
\end{lemma}

\begin{proof}
For any $0\leq t\leq T$, since
\begin{align*}
	\frac{d}{dt}\cF(u^N(x,t))=-\int_\Omega M_\theta(u^N(x,t))|\nabla\omega^N(x,t)|^2dx,
\end{align*}
integrating in time over $(0,t)$, we obtain
\begin{align} \label{ener-id}
	&\int_\Omega W(u^N(x,t)) + \kappa\left(\frac{1}{2}|u^N(x,t)|^2 - |\nabla u^N(x,t)|^2 + \frac{1}{2}|\Delta u^N(x,t)|^2\right)dx \nonumber \\ 
	&+\int_{0}^{t}\int_\Omega M_\theta(u^N(x,\tau))|\nabla\omega^N(x,\tau)|^2dxd\tau \nonumber \\ 
	&=\int_\Omega W(u^N(x,0)) + \kappa\left(\frac{1}{2}|u^N(x,0)|^2 - |\nabla u^N(x,0)|^2 + \frac{1}{2}|\Delta u^N(x,0)|^2\right)dx.
\end{align}
Using the growth condition \qref{grow-1} and Sobolev embedding for $H^2(\Omega)$ in $\R^d (d=1,2,3)$, 
we obtain the bound for the right hand side of \qref{ener-id}:
\begin{align} \label{ener-rhs}
	&\int_\Omega W(u^N(x,0)) + \kappa\left(\frac{1}{2}|u^N(x,0)|^2 - |\nabla u^N(x,0)|^2 + \frac{1}{2}|\Delta u^N(x,0)|^2\right)dx \nonumber \\ 
	&\leq \int_\Omega \left(b_3|u^N(x,0)|^{2m} + b_4 +  \frac{\kappa-\epsilon}{2}|u^N(x,0)|^2 - \kappa|\nabla u^N(x,0)|^2 + \frac{\kappa}{2}|\Delta u^N(x,0)|^2\right)dx  \nonumber \\ 
	&\leq C\left(\|u^N(x,0)\|_{H^2(\Omega)}^{2m} + \|u^N(x,0)\|_{H^2(\Omega)}^2 + 1\right)dx \nonumber \\ 
	&\leq C\left(\|u_0\|_{H^2(\Omega)}^{2m} + \|u_0\|_{H^2(\Omega)}^2 + 1\right)  \nonumber \\ 
	&\leq C.
\end{align}

Now, let us consider the first integral on the left hand side of \qref{ener-id}. 
For any number $a>0$, applying the AM–GM inequality (see \cite{Cauchy}, page 457, Theorem 17), we obtain 
\begin{align}\label{am-gm}
	b_1|u^N|^{2m} + (m-1)a \geq m(b_1|u^N|^{2m}a^{m-1})^{1/m} = m b_1^{1/m} a^{(m-1)/m}|u^N|^2.	
\end{align}
We choose the number $a$ so that $m b_1^{1/m} a^{(m-1)/m} = \kappa/2$,
that is,
\begin{align}
	a = a_0 := \left(\frac{\kappa}{2m}\right)^{m/(m-1)}b_1^{1/(m-1)}.
\end{align}
So, with the choice $a = a_0$, \qref{am-gm} implies
\begin{align}\label{ener-lhs1}
	\int_\Omega b_1|u^N|^{2m} dx \geq \int_\Omega \frac{\kappa}{2}|u^N|^2 dx - (m-1)a_0|\Omega|.
\end{align}
In addition, using integration by parts and Young's inequality, for any number $\beta>0$, we obtain
\begin{align}\label{ener-lhs2}
	-\kappa\int_\Omega |\nabla u^N|^2 dx = \kappa\int_\Omega u^N\Delta u^N dx \geq -\frac{\kappa}{2\beta}\int_\Omega |u^N|^2 dx - \frac{\kappa\beta}{2}\int_\Omega |\Delta u^N|^2 dx.
\end{align}
Combining \qref{grow-1}, \qref{ener-lhs1} and \qref{ener-lhs2} we obtain
\begin{align}
	&\int_\Omega W(u^N) + \kappa\left(\frac{1}{2}|u^N|^2 - |\nabla u^N|^2 + \frac{1}{2}|\Delta u^N|^2\right)dx  \nonumber \\ 
	&\geq \int_\Omega \left(b_1|u^N|^{2m} - b_2 + \frac{\kappa-\epsilon}{2}|u^N|^2 - \kappa|\nabla u^N|^2 + \frac{\kappa}{2}|\Delta u^N|^2\right)dx  \nonumber \\ 
	&\geq \left(\kappa - \frac{\epsilon}{2} - \frac{\kappa}{2\beta}\right) \int_\Omega|u^N|^2 dx + \frac{\kappa}{2}(1-\beta)\int_\Omega|\Delta u^N|^2 dx - (b_2 + (m-1)a_0)|\Omega|.
\end{align}
We choose the number $\beta$ so that $\kappa - \frac{\epsilon}{2} - \frac{\kappa}{2\beta} > 0$ and $1-\beta > 0$, 
that is, $\frac{\kappa}{2\kappa-\epsilon} < \beta$ and $\beta < 1$. 
Since $0<\epsilon<\kappa$, then $\frac{\kappa}{2\kappa-\epsilon} < 1$, 
hence, we can always find $\beta_0$ such that $\frac{\kappa}{2\kappa-\epsilon} < \beta_0 < 1$. 
Define
\begin{align*}
	\gamma_0 := \min \left\lbrace \kappa - \frac{\epsilon}{2} - \frac{\kappa}{2\beta_0}, \frac{\kappa}{2}(1-\beta_0) \right\rbrace >0, 
\end{align*}
then we have
\begin{align}\label{ener-lhs3}
	&\int_\Omega W(u^N) + \kappa\left(\frac{1}{2}|u^N|^2 - |\nabla u^N|^2 + \frac{1}{2}|\Delta u^N|^2\right)dx  \nonumber \\ 
	&\geq \gamma_0 \left(\int_\Omega|u^N|^2 dx + \int_\Omega|\Delta u^N|^2 dx\right) - (b_2 + (m-1)a_0)|\Omega|.
\end{align}
Combining \qref{ener-id}, \qref{ener-rhs} and \qref{ener-lhs3} we get
	\begin{align}
& \gamma_0 \bigg(\int_\Omega|u^N(x,t)|^2 dx + \int_\Omega|\Delta u^N(x,t)|^2 dx\bigg)
	\nonumber
	\\
& \qquad   + \int_{0}^{t}\int_\Omega M_\theta(u(x,\tau))|\nabla\omega^N(x,\tau)|^2dxd\tau \leq C,
	\label{big-bnd}
	\end{align}
for a.e. $t \in [0,T]$. 
The estimate \qref{bnd-M-omN} is established.

From \qref{big-bnd}, using integration by parts and Young's inequality again, we get
\begin{align}\label{bnd-grad-uN}
	\int_\Omega |\nabla u^N|^2 dx = -\int_\Omega u^N\Delta u^N dx &\leq \frac{1}{2}\int_\Omega |u^N|^2 dx + \frac{1}{2}\int_\Omega |\Delta u^N|^2 dx 
	\leq C.
\end{align}
Then we obtain the estimate \qref{bnd-uN-Linf-H2} by combining \qref{big-bnd} and \qref{bnd-grad-uN}.
	\end{proof}
	
\

Using the growth conditions \qref{grow-1}--\qref{grow-3} and Sobolev embedding for $H^2(\Omega)$ in $\R^d$ ($1 \leq d \leq 3$), 
we obtain the following corollary.
\begin{coro}\label{coro-est}
	Let $u^N$ be a solution of the system \qref{ode1}--\qref{ode3}. 
	Then, we have
	\begin{align}
		%\label{bnd-uN-L2-H2}\|u^N\|_{L^2(0,T;H^2(\Omega))}&\leq C, \\
		\label{bnd-uN-C} \|u^N\|_{L^\infty(0,T;C^{0,\alpha}(\bar{\Omega}))} &\leq C, \; \text{for any } 0 < \alpha < 1/2, \\
		\label{bnd-M-theta-Linf} \|M_\theta(u^N)\|_{L^\infty(0,T;L^\infty(\Omega))} &\leq C, \\
		\label{bnd-W'-Linf} \|W'(u^N)\|_{L^\infty(0,T;L^\infty(\Omega))} &\leq C, \\
		\label{bnd-W''-Linf} \|W''(u^N)\|_{L^\infty(0,T;L^\infty(\Omega))} &\leq C.
	\end{align}
\end{coro}

\begin{lemma}
	Let $u^N$ be a solution of the system \qref{ode1}--\qref{ode3}. 
	Then, we have
	\begin{align}
		\label{bnd-omN-H1}\|\omega^N\|_{L^2(0,T;H^1(\Omega))}&\leq C_\theta.
	\end{align}
\end{lemma}

	\begin{proof}
Since $M_\theta(u)\geq\theta$, 
by \qref{bnd-M-omN} we have
\begin{align}\label{bnd-grad-omN}
	\int_0^T\int_\Omega|\nabla\omega^N|^2dxdt \leq \frac{C}{\theta},
\end{align}
which implies \qref{bnd-omN-H1} by using Poincar\'{e}'s inequality.
	\end{proof}

\begin{lemma}
	Let $u^N$ be a solution of the system \qref{ode1}--\qref{ode3}. 
	Then, we have
	\begin{align}
		\label{bnd-uN-H5}\|u^N\|_{L^2(0,T;H^5(\Omega))}&\leq C_\theta.
	\end{align}
\end{lemma}

	\begin{proof}
Since $\Delta u^N = \sum_{j=1}^N c_j^N(t)\Delta\phi_j = -\sum_{j=1}^N c_j^N(t)\lambda_j\phi_j$, 
then from \qref{ode2} we get
\begin{align}
	\int_\Omega\omega^N \Delta u^N &= \int_\Omega(W'(u^N)\Delta u^N + \kappa u^N\Delta u^N - 2\kappa\nabla u^N\cdot\nabla\Delta u^N + \kappa\Delta u^N \Delta^2 u^N)dx  \nonumber \\ 
	&= \int_\Omega(W'(u^N)\Delta u^N - \kappa |\nabla u^N|^2 + 2\kappa|\Delta u^N|^2 - \kappa|\nabla\Delta u^N|^2)dx.
\end{align}
Then, using \qref{bnd-uN-Linf-H2}, \qref{bnd-W'-Linf} and \qref{bnd-omN-H1}, we obtain
\begin{align}\label{bnd-grad-lapl-uN}
	&\int_0^T\int_\Omega|\nabla\Delta u^N|^2 dxdt   \nonumber \\ 
	&= \frac{1}{\kappa}\int_0^T\int_\Omega(W'(u^N)\Delta u^N - \omega^N \Delta u^N - \kappa |\nabla u^N|^2 + 2\kappa|\Delta u^N|^2)dxdt    \nonumber \\ 
	&\leq \frac{1}{\kappa}(\|W'(u^N)\|^2_{L^2(\Omega_T)} + \|\omega^N\|^2_{L^2(0,T;L^2(\Omega))} + 2\kappa)\|\Delta u^N\|^2_{L^2(0,T;L^2(\Omega))}  \nonumber \\ 
	&\leq C_\theta.
\end{align}

By \qref{ode2}, we have
\begin{align}\label{wN-str}
	\omega^N = W'(u^N) + \kappa u^N + 2\kappa\Delta u^N + \kappa\Delta^2 u^N.
\end{align}
So by \qref{bnd-uN-Linf-H2}, \qref{bnd-W'-Linf} and \qref{bnd-omN-H1}, we obtain
\begin{align}\label{bnd-lapl2-uN}
	&\int_0^T\int_\Omega|\Delta^2 u^N|^2 dxdt   \nonumber \\ 
	&\leq \frac{4}{\kappa^2}\int_0^T\int_\Omega (|W'(u^N)|^2 + \kappa^2 |u^N|^2 + 4\kappa^2|\Delta u^N|^2 + |\omega^N|^2)dxdt    \nonumber \\ 
	&\leq C_\theta.
\end{align}

Taking the derivatives with respect to $x$ on both sides of \qref{wN-str}, we get
\begin{align}
	\nabla\omega^N = W''(u^N)\nabla u^N + \kappa\nabla u^N + 2\kappa\nabla\Delta u^N + \kappa\nabla\Delta^2 u^N.
\end{align}
Then, by \qref{bnd-uN-Linf-H2}, \qref{bnd-W''-Linf}, \qref{bnd-omN-H1} and \qref{bnd-grad-lapl-uN}, we obtain
\begin{align}\label{bnd-grad-lapl2-uN}
	&\int_0^T\int_\Omega|\nabla\Delta^2 u^N|^2 dxdt   \nonumber \\ 
	&\leq \frac{4}{\kappa^2}\int_0^T\int_\Omega (|\nabla\omega^N|^2 + |W''(u^N)|^2 |\nabla u^N|^2 + \kappa^2 |\nabla u^N|^2 + 4\kappa^2|\nabla\Delta u^N|^2)dxdt    \nonumber \\ 
	&\leq C_\theta.
\end{align}
Combining \qref{bnd-uN-Linf-H2}, \qref{bnd-grad-lapl-uN}, \qref{bnd-lapl2-uN} and \qref{bnd-grad-lapl2-uN}, we obtain \qref{bnd-uN-H5}.
	\end{proof}

\begin{lemma}
	Let $u^N$ be a solution of the system \qref{ode1}--\qref{ode3}, then we have
	\begin{align}
		\label{bnd-uN_t}\|\partial_tu^N\|_{L^2(0,T;H^{-2}(\Omega))}\leq C.
	\end{align}
\end{lemma}

	\begin{proof}
First, the estimate \qref{big-bnd} implies that $\|\sqrt{M_\theta(u^N)}\nabla\omega^N\|_{L^2(0,T;L^2(\Omega))}\leq C$.
Then, for any $\phi\in L^2(0,T;H^2(\Omega))$, by \qref{ode1} and the generalized H\"{o}lder's inequality, we have
\begin{align}
	\left| \int_\Omega\partial_t u^N\phi dx\right| &=\left| \int_\Omega M_\theta(u^N)\nabla\omega^N\nabla\phi dx\right|  \nonumber \\ 
	&\leq \left\|\sqrt{M_\theta(u^N)} \right\|_{L^d(\Omega)} \left\|\sqrt{M_\theta(u^N)}\nabla\omega^N \right\|_{L^2(\Omega)}\|\nabla\phi \|_{L^{2d/(d-2)}(\Omega)}dt  \nonumber \\ 
	&\leq C \left\|\sqrt{M_\theta(u^N)}\nabla\omega^N \right\|_{L^2(\Omega)}\|\nabla\phi\|_{L^{2d/(d-2)}(\Omega)}.
\end{align}
Since $H^1(\Omega)\subset\subset L^{2d/(d-2)}(\Omega)$, by \qref{bnd-M-theta-Linf} and H\"{o}lder's inequality, we have
\begin{align}
	\left| \int_0^T\int_\Omega\partial_t u^N\phi dxdt\right| &\leq  C\int_0^T \left\|\sqrt{M_\theta(u^N)}\nabla\omega^N \right\|_{L^2(\Omega)}\|\nabla\phi\|_{L^{2d/(d-2)}(\Omega)}dt  \nonumber \\  
	&\leq C \left\|\sqrt{M_\theta(u^N)}\nabla\omega^N \right\|_{L^2(0,T;L^2(\Omega))}\|\nabla\phi\|_{L^2(0,T;L^{2d/(d-2)}(\Omega))}  \nonumber \\  
	&\leq C\|\nabla\phi\|_{L^2(0,T;H^1(\Omega))}  \nonumber \\  
	&\leq C\|\phi\|_{L^2(0,T;H^2(\Omega))}.
\end{align}
This implies $\|\partial_tu^N\|_{L^2(0,T;H^{-2}(\Omega))}\leq C$.
 	\end{proof}

\subsection{Convergence of $u^N$ and the existence of a weak solution} \label{weak-sln-gnl}
In $\R^d$ with $1 \leq d \leq 3$, we have 
$H^2(\Omega)\subset\subset H^1(\Omega)\subset\subset C^{0,\alpha}(\bar{\Omega}) \hookrightarrow H^{-2}(\Omega)$ 
and $H^5(\Omega)\subset\subset H^l(\Omega)\subset\subset C^{3,\alpha}(\bar{\Omega}) \hookrightarrow H^{-2}(\Omega)$, 
for any integer  $1\leq l \leq 4$ and any $0<\alpha<1/2$. 
Hence, by the Aubin--Lions lemma (see \cite{Aubin, Simon-Lp}), we have
\begin{align*}
	&\{f\in 	L^\infty(0,T;H^2(\Omega)):\partial_tf\in L^2(0,T;H^{-2}(\Omega))\}\subset\subset C([0,T];H^1(\Omega)),\\
	&\{f\in 	L^\infty(0,T;H^2(\Omega)):\partial_tf\in L^2(0,T;H^{-2}(\Omega))\}\subset\subset C([0,T];C^{0,\alpha}(\bar{\Omega})),\\
	&\{f\in 	L^2(0,T;H^5(\Omega)):\partial_tf\in L^2(0,T;H^{-2}(\Omega))\}\subset\subset L^2(0,T;H^l(\Omega)), \\
	&\{f\in 	L^2(0,T;H^5(\Omega)):\partial_tf\in L^2(0,T;H^{-2}(\Omega))\}\subset\subset L^2(0,T;C^{3,\alpha}(\bar{\Omega})),
\end{align*}
for any integer $1\leq l \leq 4$ and any $0<\alpha<1/2$. 
Combining with the weak compactness in $L^2(0,T;H^5(\Omega))$ and $L^2(0,T;H^{-2}(\Omega)$, 
by \qref{bnd-uN-Linf-H2}, \qref{bnd-uN-H5} and \qref{bnd-uN_t}, 
there exist a subsequence of $\{u^N\}$ (not relabeled) and a function 
\begin{align*}	
	u_\theta\in L^2(0,T;H^5(\Omega))\cap L^2(0,T;C^{3,\alpha}(\bar{\Omega}))\cap C([0,T];H^1(\Omega))\cap C([0,T];C^{0,\alpha}(\bar{\Omega})) 
\end{align*}
such that as $N\to \infty$,
\begin{align}
	\label{u-wconv-H5}u^N&\rightharpoonup u_\theta \quad \text{weakly in } L^2(0,T;H^5(\Omega)), \\
	\label{u_t-wconv2}\partial_tu^N&\rightharpoonup \partial_tu_\theta \quad \text{weakly in } L^2(0,T;H^{-2}(\Omega)),\\
	\label{uN-conv1-H1}u^N&\to u_\theta \quad \text{strongly in } C([0,T];H^1(\Omega)), \\
	\label{uN-conv1-C}u^N&\to u_\theta \quad \text{strongly in } C([0,T];C^{0,\alpha}(\bar{\Omega})), \\
	\label{u-conv2-HL}u^N&\to  u_\theta \quad \text{strongly in } L^2(0,T;H^l(\Omega)) \; \text{and a.e. in } \Omega_T, \\
	\label{u-conv2-C3}u^N&\to  u_\theta \quad \text{strongly in } L^2(0,T;C^{3,\alpha}(\bar{\Omega})),
\end{align}
for any $1\leq l \leq 4$ and any $0<\alpha<1/2$. 
In addition, using \qref{bnd-uN-Linf-H2} and \qref{bnd-uN_t}, we have the following bounds for $u_\theta$:
\begin{align}
	\label{bnd-u-H2}\|u_\theta\|_{L^\infty(0,T;H^2(\Omega))} &\leq C, \\
	\label{bnd-u_t-H2}\|\partial_tu_\theta\|_{L^2(0,T;H^{-2}(\Omega))} &\leq C.
\end{align}

From \qref{bnd-omN-H1}, there exist a subsequence of $\{\omega^N\}$ (not relabeled) and a function $\omega_\theta\in L^2(0,T;H^1(\Omega))$ such that as $N\to \infty$,
\begin{align}
	\label{omN-wconv2-H1}\omega^N&\rightharpoonup \omega_\theta \quad \text{weakly in } L^2(0,T;H^1(\Omega)).
\end{align}
By \qref{uN-conv1-H1}, $M_\theta(u^N) \to  M_\theta(u_\theta)$ strongly in $C([0,T];L^d(\Omega))$, hence
\begin{align}
	\label{M-grad-omN-wconv1}M_\theta(u^N)\nabla\omega^N&\rightharpoonup M_\theta(u_\theta)\nabla\omega_\theta \quad \text{weakly in } L^2(0,T;L^{2d/(d+2)}(\Omega,\R^d))  \text{ as } N\to \infty.
\end{align}

For any $\gamma(t)\in L^2(0,T)$, because $\gamma(t)\nabla\phi_j\in L^2(0,T;L^{2d/(d-2)}(\Omega))$, 
multiplying both sides of \qref{ode1} by $\gamma(t)$, 
integrating over the time interval $(0,T)$, and taking the limits as $N\to \infty$, by \qref{u_t-wconv2} and \qref{M-grad-omN-wconv1}, we have
\begin{align}\label{PFC-w1s}
	\int_0^T\left\langle \partial_tu_\theta,\gamma(t)\phi_j(x)\right\rangle_{H^{-2}(\Omega),H^2(\Omega)}dt = -\int_0^T\int_\Omega M_\theta(u_\theta)\nabla\omega_\theta\cdot\gamma(t)\nabla\phi_j(x) dxdt.
\end{align}
For any function $\xi\in L^2(0,T;H^2(\Omega))$, since its Fourier series $\sum_{j=1}^\infty a_j(t)\phi_j(x)$ converges strongly to $\xi$ in
$L^2(0,T;H^2(\Omega))$, 
then $\sum_{j=1}^\infty a_j(t)\nabla\phi_j(x)$ converges strongly to $\nabla\xi$ in $L^2(0,T;L^{2d/(d-2)}(\Omega,\R^d))$. 
Hence, by \qref{PFC-w1s}, we have
\begin{align}\label{PFC-w2a}
	\int_0^T\left\langle \partial_tu_\theta,\xi\right\rangle_{H^{-2}(\Omega),H^2(\Omega)}dt = -\int_0^T\int_\Omega M_\theta(u_\theta)\nabla\omega_\theta\cdot\nabla\xi dxdt,
\end{align}
for all $\xi\in L^2(0,T;H^2(\Omega))$,. 

By \qref{uN-conv1-C} and \qref{bnd-W'-Linf}, we have
\begin{align}\label{W'-uN-conv2}
	W'(u^N) \to  W'(u_\theta) \quad \text{strongly in } C([0,T];L^q(\Omega))  \text{ as } N\to \infty,
\end{align}
for any $1\leq q <\infty$. 
Then, for any $\gamma(t)\in L^2(0,T)$, because $\gamma(t)\phi_j\in L^2(0,T;C(\Omega))$, 
multiplying both sides of \qref{ode2} by $\gamma(t)$, integrating over the time interval $(0,T)$, 
and taking the limits as $N\to \infty$, by \qref{u-wconv-H5} and \qref{W'-uN-conv2}, we obtain
\begin{align}\label{PFC-w2s}
	&\int_0^T\int_\Omega\omega_\theta\gamma(t)\phi_jdxdt  \nonumber \\ 
	&=\int_0^T\int_\Omega W'(u_\theta)\gamma(t)\phi_j + \kappa u_\theta\gamma(t)\phi_j - 2\kappa \nabla u_\theta \cdot \gamma(t)\nabla\phi_j + \kappa \Delta u_\theta\gamma(t)\Delta\phi_j \;dxdt.
\end{align}
For any function $\xi\in L^2(0,T;H^2(\Omega))$, since its Fourier series 	$\sum_{j=1}^\infty a_j(t)\phi_j(x)$ converges strongly to $\xi$ in $L^2(0,T;H^2(\Omega))$, 
then $\sum_{j=1}^\infty a_j(t)\nabla\phi_j(x)$ converges strongly to $\nabla\xi$ in $L^2(0,T;L^{2d/(d-2)}(\Omega,\R^d))$ 
and $\sum_{j=1}^\infty a_j(t)\Delta\phi_j(x)$ converges strongly to $\Delta\xi$ in $L^2(0,T;L^{2d/(d-2)}(\Omega))$. 
Hence, by \qref{PFC-w2s}, we obtain
\begin{align}\label{PFC-w2b}
	\int_0^T\int_\Omega\omega_\theta\xi dxdt = \int_0^T\int_\Omega  W'(u_\theta)\xi + \kappa u_\theta\xi - 2\kappa \nabla u_\theta \cdot \nabla\xi + \kappa \Delta u_\theta\Delta\xi \;dxdt,
\end{align}
for all $\xi\in L^2(0,T;H^2(\Omega))$. 
Since $u_\theta\in L^2(0,T;H^5(\Omega))$, 
then $\omega_\theta\in L^2(0,T;H^3(\Omega))$, hence, from \qref{PFC-w2b} we obtain
\begin{align}\label{om-theta-ae}
	\omega_\theta = W'(u_\theta) + \kappa u_\theta + 2\kappa \Delta u_\theta + \kappa \Delta^2 u_\theta \quad \text{a.e. in } \Omega_T,
\end{align}
and
\begin{align}\label{grad-om-theta-ae}
	\nabla\omega_\theta = W''(u_\theta)\nabla u_\theta + \kappa\nabla u_\theta + 2\kappa\nabla\Delta u_\theta + \kappa\nabla\Delta^2 u_\theta \quad \text{a.e. in } \Omega_T.
\end{align}
Combining \qref{grad-om-theta-ae} and \qref{PFC-w2a} we get \qref{PFC-w2}.

For the initial value, by \qref{ode3}, we have
\begin{align*}
	u^N(x,0) \to  u_0(x) \quad \text{strongly in }  L^2(\Omega)  \text{ as } N\to \infty.
\end{align*}
Combining with \qref{u-conv2-HL}, we have $u_\theta(x,0)=u_0(x)$, for a.e. $x\in\Omega$.

\subsection{Energy inequality}
From \qref{uN-conv1-H1}, we have $\sqrt{M_\theta(u^N)} \to  \sqrt{M_\theta(u_\theta)}$ strongly in $C([0,T];L^2(\Omega))$ as $N\to \infty$. 
Combining with \qref{omN-wconv2-H1}, we get
\begin{align}
	\sqrt{M_\theta(u^N)}\nabla\omega^N&\rightharpoonup \sqrt{M_\theta(u_\theta)}\nabla\omega_\theta \quad \text{weakly in } L^2(0,T;L^1(\Omega;\R^d))  \text{ as } N\to \infty.
\end{align}
On the other hand, by \qref{bnd-M-omN}, there exist a subsequence of $\{ \sqrt{M_\theta(u^N)}\nabla\omega^N \}$ (not relabeled) 
and a function $\Lambda_\theta \in L^2(0,T;L^2(\Omega;\R^d))$ such that
\begin{align}
	\sqrt{M_\theta(u^N)}\nabla\omega^N &\rightharpoonup \Lambda_\theta \quad \text{weakly in } L^2(0,T;L^2(\Omega;\R^d))  \text{ as } N\to \infty.
\end{align}
By the uniqueness of weak limits, we have $\Lambda_\theta=\sqrt{M_\theta(u_\theta)}\nabla\omega_\theta$, 
which implies
\begin{align}\label{sqrtM-grad-omN-wconv}
	\sqrt{M_\theta(u^N)}\nabla\omega^N&\rightharpoonup \sqrt{M_\theta(u_\theta)}\nabla\omega_\theta \quad \text{weakly in } L^2(0,T;L^2(\Omega;\R^d))  \text{ as } N\to \infty.
\end{align}
Also, by \qref{uN-conv1-C}, we have
\begin{align}\label{W-uN-conv2}
	W(u^N) \to  W(u_\theta) \quad \text{strongly in } C([0,T];L^q(\Omega))  \text{ as } N\to \infty,
\end{align}
for any $1\leq q <\infty$. 

Because $u^N$ and $\omega^N$ satisfy the energy identity \qref{ener-id}, 
using \qref{u-wconv-H5}--\qref{u-conv2-C3}, \qref{sqrtM-grad-omN-wconv} and \qref{W-uN-conv2}, 
and the weak lower semicontinuity of norms, 
by taking the limits as $N \to \infty$ in \qref{ener-id}, we obtain the energy inequality \qref{ener-ineq1}.

\subsection{Positive initial data}
In this section, we assume that the initial data $u_0(x)>0$, for all $x\in\Omega$. 
For any $0<\theta<1$, 
we define the entropy densities $\Phi:(0,\infty)\to  [0,\infty)$ and $\Phi_\theta: \R \to  [0,\infty)$ as
\begin{align*}
	\Phi''(z)=\frac{1}{M(z)}, \quad \Phi(1)=\Phi'(1)=0,
\end{align*}
and
\begin{align*}
	\Phi_\theta''(z)=\frac{1}{M_\theta(z)}, \quad  \Phi_\theta(1)=\Phi_\theta'(1)=0.
\end{align*}
Using the definitions of $M(u)$ and $M_\theta(u)$ in \qref{M(u)} and \qref{M-theta}, we obtain
\begin{align}\label{Phi(u)}
	\Phi(z) = z\ln(z)-z+1 \;  \text{ for } z>0,
\end{align}
and
\begin{align}\label{Phi-theta}
	\Phi_\theta(z) = \begin{cases}
					z\ln (z)-z+1 &, z>\theta,  \\
					\frac{1}{2\theta}z^2+(\ln(\theta)-1)z+1-\frac{\theta}{2} &, z\leq\theta. \end{cases}
\end{align}
We see that $\Phi_\theta\in C^2(\R)$, 
$0\leq\Phi_\theta(z)\leq\Phi(z)$, for all $z>0$, and $\Phi_\theta(z)=\Phi(z)$, for all $z\geq\theta$.

\textbf{Claim 1.} \textit{For any $t\in [0,T]$,}
\begin{align}\label{Phi-theta-eqn}
	\int_\Omega\Phi_\theta(u_\theta(x,t))dx - \int_\Omega\Phi_\theta(u_0(x))dx=-\int_0^t\int_\Omega \nabla\omega_\theta\cdot\nabla u_\theta dxd\tau.
\end{align}

\begin{proof}
For any $\epsilon>0$, let $\Phi_{\theta,\epsilon}$ be the mollification of $\Phi_\theta$. 
Since $\Phi_\theta \in C^2(\R)$, then $\Phi_{\theta,\epsilon} \to \Phi_\theta, \Phi'_{\theta,\epsilon} \to \Phi'_\theta$ and $\Phi''_{\theta,\epsilon} \to \Phi''_\theta$ uniformly on compact subsets of $\R$ as $\epsilon \to 0$. 
We will show that
\begin{align}\label{Phi-theta-eps-eqn}
	\int_\Omega\Phi_{\theta,\epsilon}(u_\theta(x,t))dx - \int_\Omega\Phi_{\theta,\epsilon}(u_0(x))dx=-\int_0^t\int_\Omega M_\theta(u_\theta)\Phi_{\theta,\epsilon}''(u_\theta)\nabla\omega_\theta\cdot\nabla u_\theta dxd\tau.
\end{align}

For any $h>0$, define 
\begin{align}
	u_{\theta,h}(x,t):=\frac{1}{h}\int_{t-h}^ tu_\theta(x,\tau)d\tau,
\end{align}
where we set $u_\theta(x,t)=u_0(x)$ for $t\leq 0$. 
Since $H^3(\Omega)\subset\subset H^2(\Omega)\hookrightarrow H^{-3}(\Omega)$, by the Aubin--Lions lemma,
\begin{align*}
	\{f\in 	L^2(0,T;H^3(\Omega)):\partial_tf\in L^2(0,T;H^{-3}(\Omega))\}\subset\subset L^2(0,T;H^2(\Omega)).
\end{align*}
Since $u_\theta\in L^\infty(0,T;H^5(\Omega))$ and $\Phi'_{\theta,\epsilon}, \Phi''_{\theta,\epsilon}, \Phi'''_{\theta,\epsilon}, \Phi_{\theta,\epsilon}^{(4)}$ are bounded, we have
\begin{align*}
	\sup_{h>0}\|\Phi_{\theta,\epsilon}'(u_{\theta,h})\|_{L^2(0,T;H^3(\Omega))} \leq C_\theta \quad \text{and} 
	\quad  \partial_t\Phi_{\theta,\epsilon}'(u_{\theta,h})\in L^2(0,T;H^{-3}(\Omega)),
\end{align*}
for any $h>0$. 
Hence, there exists a subsequence of $\{\Phi_{\theta,\epsilon}'(u_{\theta,h})\}_{h>0}$ (not relabeled) such that
\begin{align} \label{conv-Phi'-u-theta-h}
	\Phi_{\theta,\epsilon}'(u_{\theta,h}) \to  \Phi_{\theta,\epsilon}'(u_\theta) \quad \text{strongly in }  L^2(0,T;H^2(\Omega)) \; 
	\text{as } h\to  0.
\end{align}
On the other hand, define  $\textup{\textbf{J}}_\theta:=-M_\theta(u_\theta)\nabla\omega_\theta$. 
Since $\omega_\theta \in L^2(0,T;H^3(\Omega))$, $\textup{\textbf{J}}_\theta \in L^2(0,T;H^2(\Omega;\R^d))$. 
Hence, for any $\xi\in L^2(0,T;H^2(\Omega))$,
\begin{align}
	&\left| \left\langle \partial_tu_{\theta,h} - \partial_tu_\theta,\xi\right\rangle_{(L^2(0,T;H^{-2}(\Omega)),L^2(0,T;H^2(\Omega))} \right|  \nonumber \\ 
	=& \frac{1}{h}\left| \int_0^T\left\langle \int_{t-h}^t(\partial_tu_\theta(\tau) - \partial_tu_\theta(t))d\tau,
	\xi\right\rangle_{(H^{-2}(\Omega),H^2(\Omega))} dt \right|  \nonumber \\  
	=& \frac{1}{h}\left| \int_0^T\left\langle \int_{-h}^0(\partial_tu_\theta(t+s) - \partial_tu_\theta(t))ds,
	\xi\right\rangle_{(H^{-2}(\Omega),H^2(\Omega))} dt \right| \nonumber \\   
	\leq& \frac{1}{h} \int_{-h}^0 \left| \int_0^T \int_\Omega \nabla\xi\cdot (\textup{\textbf{J}}_\theta(t+s)
	-\textup{\textbf{J}}_\theta(t))dxdt\right| ds  \nonumber \\ 
	\leq& \|\xi\|_{L^2(0,T;H^2(\Omega))}\sup_{-h\leq s \leq 0}\|\textup{\textbf{J}}_\theta(\cdot+s)
	-\textup{\textbf{J}}_\theta(\cdot)\|_{L^2(0,T;H^2(\Omega))},
\end{align}
which implies that
\begin{align}
	\|\partial_tu_{\theta,h} - \partial_tu_\theta\|_{L^2(0,T;H^{-2}(\Omega))} \leq \sup_{-h\leq s \leq 0}\|\textup{\textbf{J}}_\theta(\cdot+s)
	-\textup{\textbf{J}}_\theta(\cdot)\|_{L^2(0,T;H^2(\Omega))}.	
\end{align}
Since $\sup_{-h\leq s \leq 0}\|\textup{\textbf{J}}_\theta(\cdot+s)
-\textup{\textbf{J}}_\theta(\cdot)\|_{L^2(0,T;H^2(\Omega))} \to 0$ as $h\to  0$, then
\begin{align} \label{conv-dt/du-theta-h}
	\partial_tu_{\theta,h} \to  \partial_tu_\theta \quad \text{strongly in } 
	L^2(0,T;H^{-2}(\Omega)) \; \text{as } h\to  0.
\end{align}
Since $\Phi_{\theta,\epsilon}'(u_{\theta,h})$ and $\partial_tu_{\theta,h}$ are both in $L^2(\Omega_T)$, we have
\begin{align}
	\int_0^t\left\langle \partial_tu_{\theta,h},\Phi_{\theta,\epsilon}'(u_{\theta,h})\right\rangle_{(H^{-2}(\Omega),H^2(\Omega))}d\tau
	&= \int_0^t \int_\Omega \Phi_{\theta,\epsilon}'(u_{\theta,h})\partial_tu_{\theta,h} dxd\tau  \nonumber \\ 
	&= \int_\Omega \int_0^t \partial_t \Phi_{\theta,\epsilon}(u_{\theta,h}(\tau,x)) d\tau dx  \nonumber \\ 
	&=\int_\Omega\Phi_{\theta,\epsilon}(u_{\theta,h}(x,t))dx - \int_\Omega \Phi_{\theta,\epsilon}(u_0(x))dx,
\end{align}
for a.e. $t\in[0,T]$.
Passing to the limit as $h\to  0$ and combining with \qref{conv-Phi'-u-theta-h} and 
\qref{conv-dt/du-theta-h}, we get
\begin{align}\label{Phi-theta-eps-lhs}
	\int_0^t\left\langle \partial_tu_\theta,\Phi_{\theta,\epsilon}'(u_\theta)\right\rangle_{(H^{-2}(\Omega),H^2(\Omega))}d\tau 
	= \int_\Omega \Phi_{\theta,\epsilon}(u_\theta(x,t))dx - \int_\Omega\Phi_{\theta,\epsilon}(u_0(x))dx.
\end{align}
Since $\Phi_{\theta,\epsilon}''$ and $\Phi_{\theta,\epsilon}'''$ are bounded, 
then $\Phi_{\theta,\epsilon}'(u_\theta)\in L^2(0,T;H^2(\Omega))$. 
So $\Phi_{\theta,\epsilon}'(u_\theta)$ is an admissible test function for the equation \qref{PFC-w2a}. 
Hence, for any $t\in [0,T]$,
\begin{align}\label{Phi-theta-eps-rhs}
	\int_0^t\left\langle \partial_tu_\theta,\Phi_{\theta,\epsilon}'(u_\theta)\right\rangle_{(H^{-2}(\Omega),H^2(\Omega))}d\tau 
	&=-\int_0^t\int_\Omega 		M_\theta(u_\theta)\nabla\omega_\theta\cdot\nabla(\Phi_{\theta,\epsilon}'(u_\theta)) dxd\tau  \nonumber \\ 
	&=-\int_0^t\int_\Omega M_\theta(u_\theta)\Phi_{\theta,\epsilon}''(u_\theta)\nabla\omega_\theta\cdot\nabla u_\theta dxd\tau.
\end{align}
Combining \qref{Phi-theta-eps-lhs} and \qref{Phi-theta-eps-rhs} we get \qref{Phi-theta-eps-eqn}.

Now, for each $t \in [0,T]$, since $u_\theta \in C([0,T];C^{3,\alpha}(\bar{\Omega}))$, 
then $u_\theta(\bar{\Omega},t):=\{u_\theta(x,t): x\in\bar{\Omega}\}$ is a compact subset of $\R$, 
so $\Phi_{\theta,\epsilon} \to \Phi_\theta$ and $\Phi''_{\theta,\epsilon} \to \Phi''_\theta$ uniformly in $u_\theta(\bar{\Omega},t)$ as $\epsilon \to 0$, 
hence, $\Phi_{\theta,\epsilon}(u_\theta) \to \Phi_\theta(u_\theta)$ and $\Phi''_{\theta,\epsilon}(u_\theta) \to \Phi''_\theta(u_\theta)$ uniformly in $\bar{\Omega}$ as $\epsilon \to 0$. 
Also, $u_\theta(\overline{\Omega_T}):=\{u_\theta(x,\tau): x\in\bar{\Omega},\tau\in[0,T]\}$ is a compact subset of $\R$, 
so $\Phi''_{\theta,\epsilon} \to \Phi''_\theta$ uniformly in $u_\theta(\overline{\Omega_T})$ as $\epsilon \to 0$, 
hence, $\Phi''_{\theta,\epsilon}(u_\theta) \to \Phi''_\theta(u_\theta)$ uniformly in $\overline{\Omega_T}$ as $\epsilon \to 0$. 
Using Lemma \ref{bnd-lem}, we obtain
\begin{align}
	\|M_\theta(u_\theta)\|_{L^\infty(\Omega_T)} <\infty \quad \text{and} \quad \left\|\sqrt{M_\theta(u_\theta)}\nabla\omega_\theta\right\|_{L^2(\Omega_T)}<\infty.
\end{align}
Since $u_\theta \in C([0,T];C^{3,\alpha}(\bar{\Omega}))$, then $\|\nabla u_\theta\|_{L^\infty(\Omega_T)} <\infty$. 
Thus,
\begin{align}
	&\left| \int_0^t\int_\Omega M_\theta(u_\theta)\nabla\omega_\theta\cdot\nabla u_\theta\left( \Phi_{\theta,\epsilon}''(u_\theta)-\Phi_\theta''(u_\theta)\right)  dxd\tau\right|  \nonumber \\ 
	\leq &\|M_\theta(u_\theta)\|_{L^\infty(\Omega_T)}^{1/2}\|\nabla u_\theta\|_{L^\infty(\Omega_T)} \|\sqrt{M_\theta(u_\theta)}\nabla\omega_\theta\|_{L^2(\Omega_T)} \|\Phi_{\theta,\epsilon}''(u_\theta)-\Phi_\theta''(u_\theta)\|_{L^2(\Omega_T)}  \nonumber \\ 
	& \to 0 \quad \mathrm{as} \; \epsilon \to 0.
\end{align}
So passing to limits as $\epsilon \to 0$ on both sides of \qref{Phi-theta-eps-eqn} and using $\Phi''_\theta(u_\theta)=1/M_\theta(u_\theta)$, we get \qref{Phi-theta-eqn}.
	\end{proof}

\

\textbf{Claim 2.} \textit{Let $(u_\theta)_-:=\min\{u_\theta,0\}$, then, for any $0<\theta<1$,}
\begin{align}\label{bnd-neg2}
	\mathop{\textup{ess sup}}_{0\leq t \leq T}\int_\Omega |(u_\theta)_-+\theta|^2dx\leq C(\theta^2+\theta+\theta^{1/2}).
\end{align}

	\begin{proof}
For any $z\leq0$, we rewrite $\Phi_\theta(z)$ as
\begin{align}\label{phi-theta-neg}
	\Phi_\theta(z) =\frac{1}{2\theta}(z+\theta)^2+(\ln\theta-2)z+(1-\theta).
\end{align}
Since $0<\theta<1$, \qref{phi-theta-neg} implies
\begin{align}
	(z+\theta)^2 \leq 2\theta \Phi_\theta(z) \quad \mathrm{for\; all} \;z\leq 0.
\end{align}
Hence, for any $t\in[0,T]$, since $\Phi_\theta(z) \geq 0$, for all $z\in\R$, we have
\begin{align}\label{bnd-u-theta-neg}
	\int_\Omega |(u_\theta(x,t))_-+\theta|^2dx &\leq 2 \theta\int_\Omega\Phi_\theta(u_\theta(x,t)_-)dx  \nonumber \\ 
	&\leq 2\theta \left( \int_{\{u_\theta\leq 0\}}\Phi_\theta(u_\theta(x,t))dx+\int_{\{u_\theta>0\}}\Phi_\theta(0)dx\right)  \nonumber \\ 
	&\leq 2\theta \left( \int_\Omega\Phi_\theta(u_\theta(x,t))dx+\int_\Omega\Phi_\theta(0)dx\right)  \nonumber \\ 
	&\leq 2\theta \left[ \int_\Omega\Phi_\theta(u_\theta(x,t))dx+\left(1-\frac{\theta}{2} \right)|\Omega|\right] 
\end{align} 

From \qref{bnd-grad-omN} and \qref{bnd-uN-Linf-H2}, we have $\|\nabla\omega_\theta\|_{L^2(\Omega_T)} \leq C/\theta^{1/2}$ 
and $\|\nabla u_\theta\|_{L^2(\Omega_T)} \leq C$. Thus, by \qref{Phi-theta-eqn} and H\"{o}lder's inequality, we have
\begin{align}\
	\left|\int_\Omega\Phi_\theta(u_\theta(x,t))dx\right| &\leq \left|\int_\Omega\Phi_\theta(u_0)dx\right| 
	+ \left|\int_0^t\int_\Omega \nabla\omega_\theta(x,\tau)\cdot\nabla u_\theta(x,\tau) dxd\tau\right|  \nonumber\\
	&\leq \int_\Omega\Phi_\theta(u_0)dx + \|\nabla\omega_\theta\|_{L^2(\Omega_T)}\|\nabla u_\theta\|_{L^2(\Omega_T)} 	\nonumber \\
	&\leq \int_\Omega\Phi(u_0)dx + \frac{C}{\theta^{1/2}},
\end{align}
for any $t\in[0,T]$. 
This inequality and \qref{bnd-u-theta-neg} implies \qref{bnd-neg2}.
	\end{proof}

\section{Proof of Theorem \ref{thm-main2}: Weak solutions for the degenerate mobility case}\label{dege-mobi}
In this section we prove the existence of a nonnegative weak solution to the PFC equation \qref{pfc-eqn1}--\qref{pfc-eqn2} 
with the degenerate mobility $M(u)$ defined by \qref{M(u)}.

\subsection{Weak convergence of $\{u_{\theta}\}$ as $\theta \to 0$}
Fix $u_0\in H^2(\Omega)$ and a sequence of positive numbers $\{\theta_i\}_{i=1}^\infty$ that monotonically decreases to 0 as $i\to \infty$. 
For each $\theta_i$, according to Theorem \ref{thm-main1}, there exists a function
\begin{align*}
	u_{\theta_i}\in L^2(0,T;H^5(\Omega))\cap L^2(0,T;C^{3,\alpha}(\bar{\Omega}))\cap C([0,T];H^1(\Omega))\cap C([0,T];C^{0,\alpha}(\bar{\Omega})), 
\end{align*}
for any $0<\alpha<1/2$, whose weak derivative is
\begin{align*}
	\partial_tu_{\theta_i}\in L^2(0,T;H^{-2}(\Omega)),
\end{align*}
such that, for any $\xi\in L^2(0,T;H^2(\Omega))$,
\begin{align}
	\label{seq-PFC-w}\int_0^T\left\langle \partial_tu_{\theta_i},\xi\right\rangle_{H^{-2}(\Omega),H^2(\Omega)}dt	&=-\int_0^T\int_\Omega M_{\theta_i}(u_{\theta_i})\nabla\omega_{\theta_i}\cdot\nabla\xi dxdt, \\
	\label{seq-om-i}
	\omega_{\theta_i}&= W'(u_{\theta_i}) + \kappa u_{\theta_i} + 2\kappa \nabla u_{\theta_i} + \kappa \Delta u_{\theta_i}, \\
	\label{ui(0)}u_{\theta_i}(x,0)&=u_0(x), \text{ for all } x\in\Omega.
\end{align}
For simplicity, we use the notations $u_i:=u_{\theta_i},\omega_i:=\omega_{\theta_i}$ and $M_i:=M_{\theta_i}$. 
By the proofs of the lemmas in Section \ref{1st-estimate}, the bounds on the right hand sides of \qref{bnd-uN-Linf-H2},  \qref{bnd-M-omN} and \qref{bnd-uN_t} depend only on $d,T,\Omega,b_1, b_2, b_3, b_4, m, \kappa,\epsilon$ and $u_0$, but not on $\theta$. 
Therefore, there exists a constant $C>0$ independent on $\{\theta_i\}_{i=1}^\infty$ such that
\begin{align}
	\label{bnd-ui-H2}\|u_i\|_{L^\infty(0,T;H^2(\Omega))} &\leq C, \\
	\label{bnd-ui_t}\|\partial_tu_i\|_{L^2(0,T;H^{-2}(\Omega))} &\leq C, \\
	\label{bnd-M-ui-grad-om-i} \left\|\sqrt{M_i(u_i)}\nabla\omega_i \right\|_{L^2(\Omega_T)} &\leq C.
\end{align}
By the Aubin--Lions lemma, we have
\begin{align*}
	\{f\in 	L^\infty(0,T;H^2(\Omega)):\partial_tf\in L^2(0,T;H^{-2}(\Omega))\}\subset\subset C([0,T];H^1(\Omega)),
\end{align*}
and
\begin{align*}
	\{f\in 	L^\infty(0,T;H^2(\Omega)):\partial_tf\in L^2(0,T;H^{-2}(\Omega))\}\subset\subset C([0,T];C^{0,\alpha}(\bar{\Omega})),
\end{align*}
for any $0<\alpha<1/2$. 
Combining with the weak compactness in $L^\infty(0,T;H^2(\Omega))$ and $L^2(0,T;H^{-2}(\Omega)$, 
there exist a subsequence of $\{u_i\}$ (not relabeled) and a function
\begin{align*}
	u\in L^\infty(0,T;H^2(\Omega))\cap C([0,T];H^1(\Omega))\cap C([0,T];C^{0,\alpha}(\bar{\Omega}))
\end{align*}
such that as $i\to \infty$,
\begin{align}
	\label{ui-w*conv-H2}u_i&\rightharpoonup u \quad \text{weakly--* in } L^\infty(0,T;H^2(\Omega)), \\	
	\label{ui_t-wconv}\partial_tu_i&\rightharpoonup \partial_tu \quad \text{weakly in } L^2(0,T;H^{-2}(\Omega)),\\
	\label{ui-conv-C}u_i&\to  u \quad \text{strongly in } C([0,T];C^{0,\alpha}(\bar{\Omega})),  \text{ for any } 0<\alpha<1/2, \\
	\label{ui-conv-W2p}u_i&\to  u \quad \text{strongly in } C([0,T];H^1(\Omega)) \; \text{and a.e. in } \Omega_T.
\end{align}

From \qref{ui-conv-C} and \qref{ui-conv-W2p}, the uniform convergence of $M_i\to  M$ and $\sqrt{M_i}\to  \sqrt{M}$ as $i\to \infty$, and the General Lebesgue Dominated Convergence Theorem (see \cite{Royden}, page 89, Theorem 19) we have
\begin{align}
	\label{Mi-conv}M_i(u_i)&\to  M(u) \quad \text{strongly in } C([0,T];L^{d/2}(\Omega)), \\
	\label{sqrtMi-conv}\sqrt{M_i(u_i)}&\to  \sqrt{M(u)} \quad \text{strongly in } C([0,T];L^d(\Omega)),
\end{align}
as $i \to \infty$. 
By \qref{bnd-M-ui-grad-om-i}, there exist a subsequence of $\{\sqrt{M_i(u_i)}\nabla\omega_i\}$ (not relabeled) and a function $\Lambda\in L^2(\Omega_T;\R^d)$ such that
\begin{align}
	\label{sqrtMi-grad-om-i-conv}\sqrt{M_i(u_i)}\nabla\omega_i\rightharpoonup \Lambda \quad \text{weakly in } L^2(\Omega_T;\R^d)  \text{ as } i \to \infty.
\end{align}
Combining with \qref{sqrtMi-conv} we get
\begin{align}
	\label{Mi-grad-om-i-conv}M_i(u_i)\nabla\omega_i\rightharpoonup \sqrt{M(u)}\Lambda \quad \text{weakly in } L^2(0,T;L^{2d/(d+2)}(\Omega;\R^d))  \text{ as } i \to \infty.
\end{align}
So, taking the limits as $i\to \infty$ in \qref{seq-PFC-w}, we obtain
\begin{align}\label{PFC-w3a}
	\int_0^T\left\langle\partial_tu,\xi\right\rangle_{H^{-2}(\Omega),H^2(\Omega)}dt=-\int_0^T\int_\Omega \sqrt{M(u)}\Lambda\cdot\nabla\xi dxdt
\end{align}
for all $\xi\in L^2(0,T;H^2(\Omega))$. 
For the initial data, from \qref{ui(0)} and \qref{ui-conv-C}, we have $u(x,0)=u_0(x)$, for all $x\in\Omega$.

Next, we will show that the function $u$ solves the PFC equation \qref{pfc-eqn1}--\qref{pfc-eqn2} in any open subset of $\Omega_T$ in which $u$ has sufficient regularity. 
Moreover, the set in which $u$ does not have sufficient regularity is contained in the set in which $M(u)$ is degenerate and another set of Lebesgue measure zero.

\subsection{Weak solution to the degenerate PFC equation}

Choose a sequence of positive numbers $\{\delta_j\}_{j=1}^\infty$ that monotonically decreases to $0$. 
For each $\delta_j$, by \qref{ui-conv-W2p} and Egorov's theorem, there exists a subset $B_j\subset\Omega_T$ with $|\Omega_T\backslash B_j|<\delta_j$ such that
\begin{align}\label{ui-conv-unif}
	u_i\to  u \quad \text{uniformly in } B_j \text{ as } i\to \infty.
\end{align}
We may choose $B_j$'s so that $B_1\subset B_2\subset ...\subset B_j\subset B_{j+1}\subset ...\subset\Omega_T$. 
Let 
\begin{align*}
	B :=\bigcup_{j=1}^{\infty} B_j,
\end{align*}
then $|\Omega_T\backslash B|=0$. We also define
\begin{align*}
	P_j:=\{(x,t)\in\Omega_T: u(x,t)>\delta_j\},
\end{align*}
then $P_1\subset P_2\subset ...\subset P_j\subset P_{j+1}\subset ...\subset\Omega_T$. 
Let
\begin{align*}
	P:=\bigcup_{j=1}^{\infty} P_j = \{(x,t)\in\Omega_T:u(x,t)>0\}.
\end{align*}
For each $j$, $B_j$ can be split into two parts:
\begin{align*}
	&B_j\cap P_j, \text{ where } u>\delta_j  \text{ and } u_i\to  u  \text{ uniformly as } i \to \infty,\\
	&B_j\backslash P_j,  \text{ where } u\leq\delta_j  \text{ and } u_i\to  u \text{ uniformly as } i \to \infty.
\end{align*}
By the choice of $B_j$ and $P_j$, we have
\begin{align}\label{incr-set-seq}
	(B_1\cap P_1)\subset (B_2\cap P_2)\subset...\subset (B_j\cap P_j)\subset (B_{j+1}\cap P_{j+1})\subset...\subset (B\cap P),
\end{align}
and
\begin{align}\label{lim-set-seq}
	B\cap P = \bigcup_{j=1}^\infty(B_j\cap P_j).
\end{align}

For any $\Phi\in L^2(0,T;L^{2d/(d-2)}(\Omega;\R^d))$ and for each $j$, we have
\begin{align}\label{3part-sum}
	\int_{\Omega_T}M_i(u_i)\nabla\omega_i\cdot\Phi dxdt=&\int_{\Omega_T\backslash B_j}M_i(u_i)\nabla\omega_i\cdot\Phi dxdt +\int_{B_j\cap P_j}M_i(u_i)\nabla\omega_i\cdot\Phi dxdt  \nonumber \\ 
	&+\int_{B_j\backslash P_j}M_i(u_i)\nabla\omega_i\cdot\Phi dxdt.
\end{align}
As $i\to \infty$, by \qref{Mi-grad-om-i-conv}, the left hand side of \qref{3part-sum} has the limit
\begin{align}\label{3part-sum-lhs}
	\lim\limits_{i\to \infty}\int_{\Omega_T}M_i(u_i)\nabla\omega_i\cdot\Phi dxdt=\int_{\Omega_T}\sqrt{M(u)}\Lambda\cdot\Phi dxdt.
\end{align}
For the first term on the right hand side of \qref{3part-sum}, since $\lim_{j\to \infty}|\Omega_T\backslash B_j| = |\Omega_T\backslash B| = 0$, 
we have
\begin{align}\label{3part-sum-rhs1}
	\lim\limits_{j\to \infty}\lim\limits_{i\to \infty}\int_{\Omega_T\backslash B_j}M_i(u_i)\nabla\omega_i\cdot\Phi dxdt=\lim\limits_{j\to \infty}\int_{\Omega_T\backslash B_j}\sqrt{M(u)}\Lambda\cdot\Phi dxdt=0.
\end{align}

We now analyze the second term on the right hand side of \qref{3part-sum}. 
Because $u_i\to  u$ uniformly in $B_1$ as $i \to \infty$, 
there exists an integer $N_1>0$ such that, for all $i\geq N_1$,
\begin{align*}
	u_i>\frac{\delta_1}{2}  \text{ in } B_1\cap P_1, \quad \text{and} \quad u_i\leq 2\delta_1 \text{ in } B_1\backslash P_1.
\end{align*}
Then by \qref{bnd-M-ui-grad-om-i}, for any $i\geq N_1$, we have
\begin{align}
	\frac{\delta_1}{2}\int_{B_1\cap P_1}|\nabla\omega_i|^2dxdt &\leq \int_{B_1\cap P_1}M_i(u_i)|\nabla\omega_i|^2 dxdt  \nonumber \\ 
	&\leq \int_{\Omega_T}M_i(u_i)|\nabla\omega_i|^2 dxdt \leq C,
\end{align}
which implies $\{\nabla\omega_i\}_{i=N_1}^\infty$ is bounded in $L^2(B_1\cap P_1;\R^d)$. 
So there exist a subsequence $\{\nabla\omega_{1,k}\}_{k=1}^\infty$ of $\{\nabla\omega_i\}_{i=N_1}^\infty$ and a function $\Psi_1\in L^2(B_1\cap P_1;\R^d)$ such that
\begin{align*}
	\nabla\omega_{1,k} \rightharpoonup\Psi_1 \quad \text{weakly in } L^2(B_1\cap P_1;\R^d) \text{ as } k\to \infty.
\end{align*}
We also write $\{u_{1,k}\}_{k=1}^\infty$ as the subsequence of $\{u_i\}_{i=N_1}^\infty$ corresponding to $\{\nabla\omega_{1,k}\}_{k=1}^\infty$. 
Using the same process, for each $j=1,2,...$, we obtain a subsequence $\{\nabla\omega_{j,k}\}_{k=1}^\infty$ of $\{\nabla\omega_{j-1,k}\}_{k=1}^\infty$ and a function $\Psi_j\in L^2(B_j\cap P_j;\R^d)$ such that
\begin{align*}
	\nabla\omega_{j,k} \rightharpoonup\Psi_j \quad \text{weakly in } L^2(B_j\cap P_j;\R^d) \text{ as } k\to \infty,
\end{align*}
and we also write $\{u_{j,k}\}_{k=1}^\infty$ as the subsequence of $\{u_{j-1,k}\}_{k=1}^\infty$ corresponding to $\{\nabla\omega_{j,k}\}_{k=1}^\infty$. 
Moreover, for any $j,k=1,2,...$, we have
\begin{align}\label{split-value-uj}
	u_{j,k}>\frac{\delta_j}{2}  \text{ in } B_j\cap P_j, \quad \text{and} \quad u_{j,k}\leq 2\delta_j \text{ in } B_j\backslash P_j.
\end{align}
Also, by \qref{incr-set-seq} and \qref{lim-set-seq}, $\Psi_j=\Psi_{j-1}$ a.e. in $B_{j-1}\cap P_{j-1}$. 
Moreover, we can extend each function $\Psi_j\in L^2(B_j\cap P_j;\R^d)$ to a function $\widehat{\Psi}_j\in L^2(B\cap P;\R^d)$ by defining
\begin{align*}
	\widehat{\Psi}_j(x,t)=\begin{cases}
		\Psi_j(x,t) &\;, \; \text{if } (x,t)\in B_j\cap P_j,\\
		0 &\;, \; \text{if } (x,t)\in (B\cap P)\backslash(B_j\cap P_j). \end{cases}.
\end{align*}
With this definition, $\lim_{j\to \infty}\widehat{\Psi}_j(x,t)$ exists for a.e. $(x,t)\in B\cap P$. 
Define
\begin{align*}
	\Psi(x,t):=\lim_{j\to \infty}\widehat{\Psi}_j(x,t) \quad \text{ for a.e. } (x,t) \in B\cap P,
\end{align*}
then $\Psi(x,t)=\Psi_j(x,t)$, for a.e. $(x,t)\in B_j\cap P_j$ and for any $j=1,2,...$.

Using a standard diagonal argument, we can extract a subsequence $\{\nabla\omega_{k,N_k}\}_{k=1}^\infty$ such that, for each $j=1,2,...$,
\begin{align}\label{grad-om-k-wconv1}
	\nabla\omega_{k,N_k} \rightharpoonup\Psi \quad \text{weakly in } L^2(B_j\cap P_j;\R^d) \text{ as } k\to \infty.
\end{align}
Combining with \qref{sqrtMi-conv}, for each $j=1,2,...$, we have
\begin{align}\label{sqrtMk-grad-wk-conv}
	\chi_{B_j\cap P_j}\sqrt{M_{k,N_k}(u_{k,N_k})}\nabla\omega_{k,N_k} \rightharpoonup \chi_{B_j\cap P_j}\sqrt{M(u)}\Psi \; \text{weakly in } L^2(0,T;L^\frac{2d}{d+2}(\Omega;\R^d))
\end{align}
as $k\to \infty$, where $\chi_{B_j\cap P_j}$ is the characteristic function of $B_j\cap P_j \subset\Omega_T$. 
Then, combining with \qref{sqrtMi-grad-om-i-conv}, we have $\Lambda=\sqrt{M(u)}\Psi$ in every set $B_j\cap P_j$, which implies that
\begin{align}
	\Lambda=\sqrt{M(u)}\Psi \quad \text{in }  B\cap P.
\end{align}
Consequently, by \qref{Mi-grad-om-i-conv}, we have, as $k \to \infty$,
\begin{align}\label{3part-sum-rhs2}
	\chi_{B\cap P}M_{k,N_k}(u_{k,N_k})\nabla\omega_{k,N_k} \rightharpoonup \chi_{B\cap P}M(u)\Psi \; \text{weakly in } L^2(0,T;L^{2d/(d+2)}(\Omega;\R^d)).
\end{align}

For the third term on the right hand side of \qref{3part-sum}, from \qref{bnd-M-ui-grad-om-i} and \qref{split-value-uj}, and the generalized H\"{o}lder's inequality, we have
\begin{align}\label{3part-sum-rhs3-est}
	&\left|\int_{B_j\backslash P_j}M_{k,N_k}(u_{k,N_k})\nabla\omega_{k,N_k}\cdot\Phi dxdt\right|  \nonumber \\ 
	&\leq \left(\sup_{B_j\backslash P_j}\sqrt{M_{k,N_k}(u_{k,N_k})}\right)
	\left\| \sqrt{M_{k,N_k}(u_{k,N_k})}\nabla\omega_{k,N_k}\right\|_{L^2(B_j\backslash P_j)}
	\|\Phi\|_{L^2(B_j\backslash P_j)}  \nonumber \\ 
	&\leq \max \{\sqrt{2\delta_j},\sqrt{\theta_{k,N_k}} \}
	\left\| \sqrt{M_{k,N_k}(u_{k,N_k})}\nabla\omega_{k,N_k}\right\|_{L^2(\Omega_T)}
	\|\Phi\|_{L^2(0,T;L^2(\Omega))}  \nonumber \\ 
	&\leq \max \{\sqrt{2\delta_j},\sqrt{\theta_{k,N_k}} \}
	\left\| \sqrt{M_{k,N_k}(u_{k,N_k})}\nabla\omega_{k,N_k}\right\|_{L^2(\Omega_T)}
	\|\Phi\|_{L^2(0,T;L^{2d/(d-2)}(\Omega))}|\Omega|^{1/d}  \nonumber \\ 
	&\leq C \max \{\sqrt{2\delta_j},\sqrt{\theta_{k,N_k}} \}.
\end{align}
Since $\lim\limits_{j\to \infty} \delta_j = 0$ and $\lim\limits_{k\to \infty} \theta_{k,N_k} = 0$, 
\qref{3part-sum-rhs3-est} implies that
\begin{align}\label{3part-sum-rhs3}
	&\lim\limits_{j\to \infty}\lim\limits_{k\to \infty} \int_{B_i\backslash P_j}M_{k,N_k}(u_{k,N_k})\nabla\omega_{k,N_k}\cdot\Phi dxdt = 0
\end{align}	

Now, in \qref{3part-sum}, replacing $u_i$ with the above subsequence $u_{k,N_k}$ and taking the limits first as $k\to \infty$ and then as $j\to \infty$, by \qref{3part-sum-lhs} \qref{3part-sum-rhs1}, \qref{3part-sum-rhs2} and \qref{3part-sum-rhs3}, we have
\begin{align}
	\int_{\Omega_T} \sqrt{M(u)}\Lambda\cdot\Phi dxdt &= \lim_{j\to \infty} \int_{B_j\cap P_j}M(u)\Psi\cdot\Phi dxdt  \nonumber \\ 
	&= \int_{B\cap P}M(u)\Psi\cdot\Phi dxdt,
\end{align}
for any $\Phi\in L^2(0,T;L^{2d/(d-2)}(\Omega;\R^d))$. 
Combining this equation with \qref{PFC-w3a}, we see that $u$ and $\Psi$ satisfy the weak formulation
\begin{align}\label{PFC-w4}
	\int_0^T\left\langle\partial_tu,\xi\right\rangle_{H^{-2}(\Omega),H^2(\Omega)}dt=-\int_{B\cap P}M(u)\Psi\cdot\nabla\xi dxdt,
\end{align}
for all $\xi\in L^2(0,T;H^2(\Omega))$.

\subsection{The relationship between $\Psi$ and $u$}
From \qref{pfc-eqn1}--\qref{pfc-eqn2} and \qref{PFC-w4}, we expect $\Psi$ to be
\begin{align*}
	\Psi = \nabla \omega = W''(u)\nabla u + \kappa\nabla u + 2\kappa\nabla\Delta u + \kappa\nabla\Delta^2 u,
\end{align*}
if the weak solution $u$ has a sufficient regularity. 
However, given the known regularities of $u$, the terms $\nabla\Delta u$ and $\nabla\Delta^2 u$ are only defined in the sense of distributions and may not even be functions. 
Therefore, we need higher regularity conditions on $u$.

\textbf{Claim.} \textit{For any open set $U\subset\Omega$ such that $\nabla\Delta^2u\in L^q(U_T)$, for some $q>1$ ($q$ may depend on $U$), where $U_T=U\times(0,T)$, we have}
\begin{align}
	\Psi=W''(u)\nabla u + \kappa\nabla u + 2\kappa\nabla\Delta u + \kappa\nabla\Delta^2 u \quad in \; U_T.
\end{align}

	\begin{proof}
Let $U\subset\Omega$ be an open set such that $\nabla\Delta^2u\in L^q(U_T)$, for some $q>1$. 
Let us consider the limit of 
\begin{align}\label{grad-om-k}
	\nabla\omega_{k,N_k}=W''(u_{k,N_k})\nabla u_{k,N_k} + \kappa\nabla u_{k,N_k} + 2\kappa\nabla\Delta u_{k,N_k} + \kappa\nabla\Delta^2 u_{k,N_k}
\end{align}
as $k\to \infty$. 
Since $\nabla\Delta^2u\in L^q(U_T)$ and $u\in L^\infty(0,T;H^2(\Omega))\cap C([0,T];H^1(\Omega))\cap C([0,T];C^{0,\alpha}(\bar{\Omega}))$, 
for any $0<\alpha<1/2$, then we have
\begin{align*}
	\Delta^2u &\in L^q(0,T;W^{1,q}(U)), \quad 
	\nabla\Delta u \in L^q(0,T;W^{2,q}(U)), \\
	\Delta u &\in L^q(0,T;W^{3,q}(U))\cap L^\infty(0,T;L^2(U)), \\
	\nabla u &\in L^q(0,T;W^{4,q}(U))\cap L^\infty(0,T;H^1(U))\cap C([0,T];L^2(U)), \\
	u &\in L^q(0,T;W^{5,q}(U))\cap L^\infty(0,T;H^2(U))\cap C([0,T];H^1(U))\cap C([0,T];C^{0,\alpha}(\bar{U})),
\end{align*}
for any $0<\alpha<1/2$.

By \qref{ui-w*conv-H2}, we have, as $k\to \infty$,
\begin{align}
	\label{grad-uk-wconv}\nabla u_{k,N_k}&\rightharpoonup \nabla u \quad \text{weakly--* in } L^\infty(0,T;H^1(U)),\\
	\label{grad-lapl-uk-wconv}\nabla\Delta u_{k,N_k}&\rightharpoonup \nabla\Delta u \quad \text{weakly--* in } L^\infty(0,T;(H^1(U))'), \\
	\label{grad-lapl2-uk-wconv}\nabla\Delta^2 u_{k,N_k}&\rightharpoonup \nabla\Delta^2 u \quad \text{weakly--* in } L^\infty(0,T;(H^3(U))').
\end{align}
We see that the bound on the right hand side of \qref{bnd-W''-Linf} depends only on $d,T,\Omega$, $b_1, b_2, b_3, b_4$, $m,\kappa,\epsilon$ and $u_0$, but not on $\theta$. 
Combining with \qref{ui-conv-W2p}, we obtain
\begin{align}
	W''(u_{k,N_k}) &\to W''(u) \quad \text{strongly in } C([0,T]; L^2(U))  \text{ as } k \to \infty.
\end{align}
Combining with \qref{grad-uk-wconv}, we have
\begin{align}\label{W''-graduk-wk-wconv}
	W''(u_{k,N_k})\nabla u_{k,N_k}\rightharpoonup W''(u)\nabla u \quad \text{weakly--* in } L^\infty(0,T;L^1(U))  \text{ as } k \to \infty.
\end{align}
Then, using \qref{grad-om-k}, \qref{grad-uk-wconv}, \qref{grad-lapl-uk-wconv}, \qref{grad-lapl2-uk-wconv} and \qref{W''-graduk-wk-wconv}, we obtain
\begin{align*}
	\nabla\omega_{k,N_k}\rightharpoonup W''(u)\nabla u + \kappa\nabla u + 2\kappa\nabla\Delta u + \kappa\nabla\Delta^2 u \quad \text{weakly--* in } L^\infty(0,T;(H^3(U))')
\end{align*}
as $k \to \infty$. 
Combining this with \qref{grad-om-k-wconv1}, by the uniqueness of the weak limit, we get
\begin{align*}
	\Psi=W''(u)\nabla u + \kappa\nabla u + 2\kappa\nabla\Delta u + \kappa\nabla\Delta^2 u \quad\text{in } B\cap P\cap U_T.
\end{align*}
Because $\Psi$ is originally defined only in $B\cap P$, 
we may extend it to $U_T$ by defining 
\begin{align*}
	\Psi : = W''(u)\nabla u + \kappa\nabla u + 2\kappa\nabla\Delta u + \kappa\nabla\Delta^2 u \quad \text{in } U_T\backslash(B\cap P). 
	\end{align*}
The claim is established.
	\end{proof}

\

Now define the set
\begin{align*}
	\cA:=\bigcup\{U_T=U\times(0,T):& \; U  \text{ is open in }\Omega \text{ and } \nabla\Delta^2u\in L^q(U_T),  \\
	&\text{for some } q>1 \text{ that may depend on } U\},
\end{align*}
then $\cA$ is open in $\Omega_T$ and
\begin{align*}
	\Psi=W''(u)\nabla u + \kappa\nabla u + 2\kappa\nabla\Delta u + \kappa\nabla\Delta^2 u \quad\text{in } \cA.
\end{align*}
So $\Psi$ is defined in $(B\cap P)\cup\cA$. To extend $\Psi$ to $\Omega_T$, notice that
\begin{align*}
	\Omega_T\backslash((B\cap P)\cup\cA)\subset \Omega_T\backslash(B\cap P)=(\Omega_T\backslash B)\cup(\Omega_T\backslash P).
\end{align*}
Since $|\Omega_T\backslash B|=0$ and $M(u)=0$ in $\Omega_T\backslash P$, 
the value of $\Psi$ outside of $(B\cap P)\cup\cA$ does not contribute to the integral on the right hand side of \qref{PFC-w4}, 
so we may define $\Psi \equiv 0$ in  $\Omega_T\backslash((B\cap P)\cup\cA)$.

\subsection{Energy inequality}
By the energy inequality \qref{ener-ineq1}, for any $j,k=1,2,...$ and any $t \in [0,T]$, we have
\begin{align}\label{ener-ineq3}
	&\int_\Omega W(u_{k,N_k}(x,t)) + \kappa\left(\frac{1}{2}|u_{k,N_k}(x,t)|^2 - |\nabla u_{k,N_k}(x,t)|^2 + \frac{1}{2}|\Delta u_{k,N_k}(x,t)|^2\right) dx   \nonumber \\ 
	&+\int_{\Omega_t\cap B_j\cap P_j} M_\theta(u_{k,N_k}(x,\tau))|\nabla\omega_{k,N_k}(x,\tau))|^2dxd\tau  \nonumber \\ 
	&\leq\int_\Omega W(u_0) + \kappa\left(\frac{1}{2}u_0^2 - |\nabla u_0|^2 + \frac{1}{2}|\Delta u_0|^2\right) dx.
\end{align}
Using \qref{ui-conv-C}, \qref{ui-conv-W2p} and \qref{sqrtMk-grad-wk-conv}, 
by taking the limits as $k\to \infty$ first and then $j\to \infty$ in \qref{ener-ineq3}, we obtain the energy inequality \qref{ener-ineq2}.

\subsection{Nonnegative weak solution with positive initial data}
Assume that the initial data $u_0(x)>0$, for all $x\in\Omega$. 
By \qref{bnd-neg1}, there exists a constant $C$ independent on $\{\theta_i\}_{i=1}^\infty$ such that, for each $i=1,2,...$,
\begin{align}\label{bnd-neg3}
	\mathop{\textup{ess sup}}_{0\leq t \leq T}\int_\Omega|(u_i(x,t))_-+\theta_i|^2dx\leq C(\theta_i^2+\theta_i+\theta_i^{1/2}).
\end{align}
Passing to the limits as $i\to \infty$ in \qref{bnd-neg3}, 
by the convergence in \qref{ui-conv-C} and \qref{ui-conv-W2p},  
we get $u\geq 0$ a.e. in $\Omega_T$. 
Moreover, since $u_0>0$ in $\Omega$, $u$ is not always zero in $\Omega_T$ owing to the continuity of $u$ in $\Omega_T$. 
This completes the proof of Theorem \ref{thm-main2}.

\section{Acknowledgments}

This work was supported by the US National Science Foundation (NSF) under the grant NSF-DMS 2309547 (Steve Wise). 

\section{Conflict of Interest}
The authors declare that they have no known conflicts of interest relevant to this article.
	
\bibliographystyle{plain}
\bibliography{Luong-bib}

\end{document}